\documentclass[12 pt]{amsart}
\usepackage[english]{babel}
\usepackage{amsfonts}
\usepackage{amsthm}
\usepackage{amsbsy}
\usepackage{amssymb}
\usepackage{graphicx}
\usepackage{eso-pic}
\usepackage{tikz}
\usepackage{xfrac}
\usepackage{float}
\usepackage[bookmarks=true,colorlinks=true,urlcolor=black,
linkcolor=black,citecolor=black,pagebackref=false]{hyperref}
\usepackage{bookmark}
\usepackage{amsmath}
\usepackage{subfigure}
\usepackage{amscd}
\usepackage{color}
\usepackage{scalerel,stackengine}
\usetikzlibrary{matrix,arrows}
\usepackage[letterpaper,margin=3.6cm,top=1in,bottom=1in]{geometry}
\usepackage[normalem]{ulem}
\usepackage{mathrsfs}

\makeatletter
\newcommand\footnoteref[1]{\protected@xdef\@thefnmark{\ref{#1}}\@footnotemark}
\makeatother

\newtheorem{lemma}{Lemma}[section]
\newtheorem{thm}[lemma]{Theorem}
\newtheorem{prop}[lemma]{Proposition}
\newtheorem{cor}[lemma]{Corollary}
\newtheorem*{cor*}{Corollary}

\newtheorem{claim}[lemma]{Claim}
\theoremstyle{definition}
\newtheorem{defn}[lemma]{Definition}

\theoremstyle{remark}
\newtheorem{rem}[lemma]{Remark}

\newcommand{\matR} {\ensuremath {\mathbb{R}}}
\newcommand{\matQ} {\ensuremath {\mathbb{Q}}}
\newcommand{\matZ} {\ensuremath {\mathbb{Z}}}

\newcommand{\matH} {\ensuremath {\mathbb{H}}}

\newcommand{\calC} {\ensuremath {\mathcal{C}}}

\newcommand{\calD} {\ensuremath {\mathcal{D}}}

\newcommand{\calH}{\ensuremath {\mathcal{H}}}

\newcommand{\rma}{\ensuremath {\mathrm{a}}}
\newcommand{\rmb}{\ensuremath {\mathrm{b}}}

\newcommand{\rmu}{\ensuremath {\mathrm{u}}}
\newcommand{\rmv}{\ensuremath {\mathrm{v}}}
\newcommand{\rmx}{\ensuremath {\mathrm{x}}}

\newcommand{\w}[1]{{\color{white} #1}}

\newcommand{\SO}{\ensuremath {\mathrm{SO}}}
\newcommand{\PO}{\ensuremath {\mathrm{PO}}}

\newcommand{\GL}{\ensuremath {\mathrm{GL}}}
\newcommand{\Or}{\ensuremath {\mathrm{O}}}

\newcommand{\Isom}{\ensuremath {\mathrm{Isom}}}

\newcommand{\Stab}{\mathrm{Stab}}
\stackMath
\newcommand\reallywidehat[1]{%
\savestack{\tmpbox}{\stretchto{%
  \scaleto{%
    \scalerel*[\widthof{\ensuremath{#1}}]{\kern-.6pt\bigwedge\kern-.6pt}%
    {\rule[-\textheight/2]{1ex}{\textheight}}%WIDTH-LIMITED BIG WEDGE
  }{\textheight}% 
}{0.5ex}}%
\stackon[1pt]{#1}{\tmpbox}%
}

\author{Alexander Kolpakov}
\address{Institut de Math\'ematiques \\ Universit\'e de Neuch\^atel \\ Rue Emile-Argand 11, CH-2000 Neuch\^atel (Switzerland)}
\email{kolpakov dot alexander at gmail dot com}

\author{Stefano Riolo}
\address{Dipartimento di Matematica \\ Universit\`a di Bologna \\ Piazza Porta San Donato 5, I-40126 (Italy)}
\email{stefano\w{a}dot\w{b}riolo\w{c}at\w{d}unibo\w{e}dot\w{f}it}

\author{Leone Slavich}
\address{Dipartimento di Matematica ``F. Casorati'' \\ Universit\`a di Pavia \\ Via Ferrata 5, I-27100 Pavia (Italy)}
\email{leone dot slavich at gmail dot com}

\thanks{A. K. and S. R. were supported by the SNSF project no. PP00P2-170560.}

\title{Embedding non-arithmetic hyperbolic manifolds}

\begin{document}

\begin{abstract}
This paper shows that many hyperbolic manifolds obtained by glueing arithmetic pieces embed into higher-dimensional hyperbolic manifolds as codimension-one totally geodesic submanifolds. As a consequence, many Gromov--Pyatetski-Shapiro and Agol--Belolipetsky--Thomson non-arithmetic  manifolds embed geodesically. Moreover, we show that the number of commensurability classes of hyperbolic manifolds with a representative of volume $\leq v$ that bounds geometrically is at least $v^{Cv}$, for $v$ large enough. 
\end{abstract}

\maketitle

\section{Introduction}

A complete finite-volume hyperbolic $n$-manifold $M$ \emph{embeds geodesically} if it can be realised as a totally geodesic embedded submanifold of a complete finite-volume hyperbolic $(n+1)$-manifold $X$. 

There are two main tools known so far to prove that a given manifold as above embeds: first, arithmetic techniques such as those used in \cite{GL,GPS,KRS,LR2,R}, and, second, explicit geometric and combinatorial constructions using Coxeter polytopes as in \cite{KMT,MZ,Mar,S1,S2}. The manifolds which are shown to embed geodesically in those papers are arithmetic. Some non-arithmetic $3$-manifolds which embed geodesically are produced in \cite{KR} by means of a right-angled hyperbolic $4$-polytope. In this paper we show that many non-arithmetic manifolds of arbitrary dimension embed geodesically.

A \emph{piece} $P$ of a hyperbolic manifold $M=\matH^n/\Gamma$ is a complete, connected hyperbolic $n$-manifold with totally geodesic boundary obtained by cutting $M$ open along a collection of pairwise disjoint, embedded, totally geodesic hypersurfaces $S_1, \ldots, S_m$. %\subset M$.

%Let $P_j$ be a piece of $M_j=\matH^n/{\Gamma_j}$, $j=1,\ldots,s$ (possibly $s=1$), and $M$ be a complete finite-volume hyperbolic manifold obtained by glueing the boundary components of $P_1, \ldots, P_s$ in pairs via isometries. Let $k$ be a totally real number field, and suppose that each $\Gamma_j$ is an arithmetic lattice of simplest type whose quadratic form $f_j$ is defined over $k$ (c.f. Section \ref{sec:simplest_type}). We prove the following:
Let us fix a totally real number field $k$. Let $M_j=\matH^n/{\Gamma_j}$, $j=1,\ldots,s$ (possibly $s=1$), be an arithmetic hyperbolic manifold of simplest type with quadratic form $f_j$ defined over $k$ (c.f. Section \ref{sec:simplest_type}). Let $P_j$ be a piece of $M_j$, and $M$ be a complete finite-volume hyperbolic manifold obtained by glueing the boundary components of $P_1, \ldots, P_s$ in pairs via isometries.

We prove the following:

\begin{thm}\label{thm:main}
If each $\Gamma_j$ is contained in $\Or(f_j,k)$, then $M$ embeds geodesically. If $M$ is orientable, the manifold into which it embeds can be chosen to be orientable.
\end{thm}

Theorem \ref{thm:main} is an extension of the recent result by Reid and two of the authors \cite{KRS}, where $s = 1$ and $M = M_1 = P_1$. 

Under the hypothesis above, we say that $M$ admits a decomposition into arithmetic pieces. If such a decomposition has more than one piece, the manifold $M$ is usually non-arithmetic. Indeed, Theorem \ref{thm:main} applies to many of the Gromov--Piatetski-Shapiro non-arithmetic manifolds \cite{GPS} (several explicit 2- and 3- dimensional examples can be constructed, c.f. \cite{Riolo-Slavich-na} for a 4-dimensional one)
and their generalisations \cite{GL,M2,R,Vinberg-na}, as well as to the ones introduced by Agol \cite{Agol} and Belolipetsky--Thomson \cite{BT} (c.f. also \cite{M1}).

In the latter case $M$ is always ``quasi-arithmetic'' (c.f. \cite{M1, T, V} for this notion), in contrast to the former case \cite{T}. In both cases, there are infinitely many commensurability classes of such manifolds \cite{R,T}, and thus we have:

\begin{cor}\label{cor:main}
There are infinitely many pairwise incommensurable non-arith\-me\-tic hyperbolic manifolds of any dimension $n\geq2$ that
embed geodesically. They can be chosen to be closed or cusped, quasi-arithmetic or not, in any combination.
\end{cor}

%\red{The manifold $X$ that we construct into which $M$ embeds geodesically is arithmetic, quasi-arithmetic, or not, if and only if $M$ is so.}

A non-trivial property for manifolds which embed geodesically is to bound geometrically. A complete (orientable) hyperbolic manifold $M$ of finite volume \emph{bounds geometrically} if it is isometric to $\partial W$, for a complete (orientable) hyperbolic manifold $W$ of finite volume with totally geodesic boundary. If $M$ bounds geometrically a manifold $W$, it clearly embeds geodesically in the double of $W$.

Despite the fact that %``most'' hyperbolic $3$-manifolds do not
to bound geometrically is a very strong requirement \cite{KRR, LR1}, for any $n\geq2$ there is a constant $c>0$ such that the number $\beta_n(v)$ of $n$-dimensional geometric boundaries of volume $\leq v$ is at least $v^{cv}$, for $v$ sufficiently big \cite{CK}. For $n\geq4$, the number $\mu_n(v)$ of all hyperbolic $n$-manifolds with volume $\leq v$ satisfies $v^{cv}\leq\mu_n(v)\leq v^{dv}$ for $v$ large enough \cite{BGLM}, so that $\beta_n$ and $\mu_n$ have the same the growth rate
(while usually $\mu_n(v) = \infty$ for $n = 2$ or $3$). The geometric boundaries constructed in \cite{CK} are arithmetic. The same lower bound is provided for the number of non-arithmetic $3$-manifolds that bound geometrically, and for the number of $4$-manifolds with connected geodesic boundary by virtue of an explicit construction \cite{KR}.

Theorem \ref{thm:main} allows us to improve significantly such considerations on geometrically bounding manifolds. Indeed, let $C_n(v)$ denote the number of commensurability classes of hyperbolic $n$-manifolds admitting a representative of volume $\leq v$, and $B_n(v)$ be the number of such classes represented by a geometric boundary of volume $\leq v$. Of course $B_n(v) \leq C_n(v) \leq \mu_n(v)$. As shown by Gelander and Levit \cite{GL}, for all $n\geq2$ we have $C_n(v) \geq v^{cv}$, for $v$ large enough. Following their arguments and applying Theorem \ref{thm:main}, we prove:

\begin{thm}\label{thm:counting}
For every $n \geq 2$, there exists $c>0$ such that $B_n(v) \geq v^{cv}$ for $v$ sufficiently large.
\end{thm}

Thus, there is plenty of geometric boundaries in any dimension, and for $n\geq4$ the growth rate of their commensurability classes is roughly the same of that of all hyperbolic $n$-manifolds. An analogous statement holds when restricting the count to either cusped or closed manifolds. In the latter case, it holds for with the extra requirement that each $M$ geometrically bounds a compact $W$. 

The manifolds that we build in order to prove Theorem \ref{thm:counting} are non-arithmetic. Indeed, there is an upper bound of the form $v^{b(\log v)^{\epsilon}}$ (and $v^b$ in the compact case) for the growth rate of commensurabilty classes of arithmetic hyperbolic manifolds of any dimension $n\geq2$ \cite{Bel, BGLS}. In other words, ``most'' hyperbolic manifolds are non-arithmetic.

\subsection*{On the proof}

The proof of Theorem \ref{thm:main} can be rougly resumed as follows: we embed the pieces into which the $n$-manifold $M$ decomposes into $(n+1)$-dimensional pieces in such a way that the latter can be glued back together.

More precisely, let $S_1, \ldots, S_{m_j}$ be the hypersurfaces of $M_j$ that produce the piece $P_j$. We show that each $M_j$ embeds geodesically in an $(n+1)$-manifold $X_j$ in such a way that $M_j$ intersects in $X_j$ orthogonally a finite collection of pairwise disjoint embedded totally geodesic hypersurfaces $Y_1, \dots, Y_{m_j}$ of $X_j$ with $Y_i \cap M_j = S_i$ (c.f. Figure \ref{fig:proof}, right).

By cutting $X_j$ open along $Y_1,\dots, Y_{m_j}$, we obtain an $(n+1)$-dimensional piece $Q_j$ in which $P_j$ is totally geodesically embedded, and intersects $\partial Q_j$ orthogonally with $P_j \cap \partial Q_j = \partial P_j$.

By carefully performing this construction for each $j = 1, \ldots, s$, we can ensure that the isometries between the boundary components of the original pieces $P_1, \dots, P_s$ extend to isometries between the boundary components of $Q_1, \dots, Q_s$. By glueing these pieces together according to the respective isometries, we produce a hyperbolic $(n+1)$-manifold $X$ into which $M$ embeds geodesically. In both the present paper and in \cite{KRS}, the main difficulties arise when proving that the manifolds considered embed without the need to pass to a finite index cover.

The two main tools which we employ are the embedding theorem from \cite{KRS} (c.f. Theorem \ref{thm:embedding-arithmetic}) for arithmetic hyperbolic manifolds of simplest type, together with the crucial fact that arithmetic hyperbolic lattices of simplest type are separable on geometrically finite subgroups (c.f. Theorem \ref{thm:GFERF}), as follows from the work \cite{BHW} by Bergeron, Haglund and Wise. We point out that the separability Theorem \ref{thm:GFERF} is used in \cite{KRS} to prove the embedding Theorem \ref{thm:embedding-arithmetic}.

In order to use the results of \cite{BHW}, we need to show that the fundamental group of the ``abstract glueing''
$ M_j\,\cup_{S_1}\,Y_1\,\cup_{S_2}\,\ldots\,\cup_{S_{m_j}}\,Y_{m_j} $,
contains a geometrically finite subgroup in which $\pi_1(M_j)$ injects, once we pass to finite-index subgroups of some $\pi_1(Y_k)$, $k = 1, \ldots, m_j$. We provide a geometric proof of this fact, which requires a more careful argument than the one given in \cite[Lemma 7.1]{BHW}.

The counting of geometric boundaries in Theorem \ref{thm:counting} basically follows by applying Theorem \ref{thm:main} to the arguments of Gelander and Levit: we glue pieces as prescribed by some decorated graphs, whose number grows super-exponentially in function of the bound on the number of vertices. The resulting manifolds embed geodesically by Theorem \ref{thm:main}. To conclude, we need to show that each of these manifolds $M$ can be chosen so to admit a fixed-point-free, orientation-reversing, isometric involution $\iota$. Indeed, if $M$ embeds geodesically in an orientable $X$, a priori we cannot ensure that $M$ disconnects $X$ (so that $M$ bounds geometrically). If it is not the case, by cutting $X$ along $M$ and quotienting out one of the two resulting boundary components by $\iota$, we have that $M$ bounds geometrically.

\subsection*{Structure of the paper}

In Section \ref{sec:preliminaires} we briefly review arithmetic manifolds of simplest type and state  Theorems \ref{thm:GFERF} and \ref{thm:embedding-arithmetic}. In Section \ref{sec:embedding-boundary} we prove Proposition \ref{prop:main}, which is the key ingredient for the proof of Theorems \ref{thm:main} and \ref{thm:counting}. The latter are proved in Section \ref{sec:proof-main}. We conclude the paper by Section \ref{sec:rmk}, with some comments about manifolds that do not embed geodesically.

\subsection*{Acknowledgements}

The authors are grateful to Jean Raimbault (Institut de Math\'ematiques de Toulouse) for stimulating discussions on the topic. A.K. and L.S. enjoyed the hospitality and atmosphere of the Oberwolfach Mini-Workshop ``Reflection Groups in Negative Curvature'' (1915b) in April 2019, during which some parts of this paper were discussed. L.S. would like to thank the Department of Mathematics at the University of Neuch\^atel for hospitality during his stay in March 2019.

\section{Preliminaries}\label{sec:preliminaires}

With a slight abuse of notation, let $J_n$ denote both the quadratic form $-x_0^2+x_1^2+\cdots +x_n^2$ over $\mathbb{R}^{n+1}$, as well as the associated diagonal matrix. We identify the hyperbolic space ${\matH}^n$ with the upper half-sheet $\{x\in {\matR}^{n+1} : J_n(x) = -1, x_0>0\}$ of the hyperboloid $\{x\in {\matR}^{n+1} : J_n(x) = -1\}$ and, by letting
$\Or(n,1)=\{A\in\GL(n+1,\matR) : A^tJ_nA=J_n\}$,
also identify $\Isom({\matH}^n)$ with the index two subgroup $\Or^{+}(n,1) < \Or(n,1)$ preserving the upper half-sheet of the hyperboloid.

\subsection{Arithmetic manifolds of simplest type} \label{sec:simplest_type}

Let $k$ be a totally real algebraic number field, together with a fixed embedding into $\mathbb{R}$ which we refer to as the identity embedding. Let $R_k$ denote the ring of integers of $k$. Let $V$ be an $(n+1)$-dimensional vector space over $k$ (by choosing a basis, we can assume $V=k^{n+1}$), equipped with a non-degenerate quadratic form $f$ defined over $k$.

We say that the form $f$ is \emph{admissible} if it has signature $(n,1)$ at the identity embedding, and signature $(n+1,0)$ at all remaining Galois embeddings of $k$ into $\mathbb{R}$. Under the assumptions above, the form $f$ is equivalent over $\matR$ to the quadratic form $J_n$, and for any non-identity Galois embedding $\sigma\colon k\rightarrow \matR$, the quadratic form $f^\sigma$ (obtained by applying $\sigma$ to each coefficient of $f$) is equivalent over $\matR$  to $x_0^2+\cdots +x_n^2$. An \emph{arithmetic subgroup} of $\Or(f,\matR)$ is a subgroup $\Gamma<\Or(f,\matR)$ commensurable (in the wide sense) with $\Or(f, R_k)$.

In order to define arithmetic subgroups of $\Or^{+}(n,1)$ we notice that, given an admissible quadratic form $f$ over $k$ of signature $(n,1)$, there exists $T\in \GL(n+1,\matR)$ such that $T^{-1}\Or(f,\matR)T = \Or(n,1)$. A subgroup $\Gamma < \Or^+(n,1)$ is called {\em arithmetic of simplest type} if $\Gamma$ is commensurable with the image in $\Or(n,1)$ of an arithmetic subgroup of $\Or(f,\matR)$ under the conjugation map above. A hyperbolic manifold $M={\matH}^n/\Gamma$ is called {\em arithmetic of simplest type} if $\Gamma$ is so.

\subsection{Immersed hypersurfaces}

Let us fix an admissible quadratic form $f$ defined over $k$. By interpreting $f$ as a form of signature $(n,1)$ on $\mathbb{R}^{n+1}=k^{n+1}\otimes \mathbb{R}$, we identify the hyperbolic space $\mathbb{H}^n$ with the appropriate half-sheet of the hyperboloid $\{x\in {\matR}^{n+1} : f(x) = -1\}$, and the group of isometries $\Isom(\mathbb{H}^n)$ with $\Or^+(f,\matR)$. A vector $v$ in $\matR^{n+1}$ is said to be a \emph{$k$-vector} if it lies in $k^{n+1}$. Given a $k$-vector $v$, we say that $v$ is \emph{space-like} if $f(v)>0$. Given a space-like $k$-vector $v$, let us denote by $v^\perp$ the subspace $\{w \in \matR^{n+1}:b_f(w,v)=0\}$, where $b_f$ denotes the symmetric bilinear form associated with $f$. Let $H_v$ denote the intersection $v^\perp \cap \matH^n$, which is a totally geodesic subspace of $\matH^n$, isometric to $\matH^{n-1}$. 

If $\Gamma$ is an arithmetic subgroup of $\Or(f,\matR)$, it is easy to see that the stabiliser of $H_v$ in $\Gamma$ is itself an arithmetic group of simplest type acting on $H_v$. We simply restrict the form $f$ to $v^{\perp}$ and notice that it is still admissible and defined over the same field $k$. We call totally geodesic subspaces of $\matH^n$ of the form $H_v$, where $v$ is a space-like $k$-vector, \emph{$\Gamma$-hyperplanes}. If the group $\Gamma$ is torsion-free, so that $M=\matH^n/\Gamma$ is a manifold, the image of $H_v$ in $M$ will be a totally geodesic, properly \textit{immersed} hypersurface with fundamental group isomorphic to $\mathrm{Stab}_{\Gamma}(H_v)$. Vice versa, every properly immersed totally geodesic hypersurface of $M$ can be constructed in this way (c.f. \cite[Corollary 5.11]{BBKS}).

\subsection{Embedding and separability} \label{sec:emb_sep}

In this section, we introduce two results about arithmetic manifolds of simplest type that will be put to essential use later on. The first one, due to Bergeron, Haglund and Wise \cite{BHW}, concerns separability of geometrically finite subgroups in arithmetic lattices of simplest type.

Let $\Gamma$ be a discrete subgroup of $\Isom(\matH^n)$. A finitely generated subgroup $G < \Gamma$ is \emph{separable} in $\Gamma$ if for every $g \in \Gamma \smallsetminus G$ there exists a finite-index subgroup $\Gamma' < \Gamma$ such that $G < \Gamma'$ and $g \not\in \Gamma'$. The group $\Gamma$ is \emph{geometrically finite extended residually finite} (``GFERF'' for short) if any geometrically finite subgroup $G<\Gamma$ is separable in $\Gamma$. 

\begin{thm}\label{thm:GFERF}
Hyperbolic arithmetic lattices of simplest type are GFERF.
\end{thm}

Since all the groups we will deal with are finitely generated, we define a geometrically finite group as one such that $\mathrm{Vol}(N_{\epsilon}(C(\Gamma)))<\infty$, where $C(\Gamma)$ is the convex core of $\mathbb{H}^n/\Gamma$. By \cite[p. 289]{Bow}, this condition is equivalent to the existence of a (possibly non-connected) finite-sided fundamental polyhedron for the action of $\Gamma$ on $\mathbb{H}^{n}$.

Now, let $M=\matH^n/\Gamma$ be an arithmetic manifold of simplest type, and let $H$ be a $\Gamma$-hyperplane. As mentioned previously, there exists a $\pi_1$-injective immersion of the manifold $S=H/\mathrm{Stab}_\Gamma(H)$ into $M$. The stabiliser of $H$ in $\Gamma$ is easily seen to be geometrically finite subgroup of $\Gamma$, and is therefore separable in $\Gamma$ by Theorem \ref{thm:GFERF}. This fact was already well known without need of Theorem \ref{thm:GFERF}; c.f. \cite{B,L}.  As a consequence, there exists a finite-index subgroup $\Gamma'<\Gamma$ such that $\mathrm{Stab}_{\Gamma}(H)<\Gamma'$, and such that $S$ lifts to a totally geodesic \emph{embedded} hypersurface in $M'=\mathbb{H}^n/\Gamma'$.

The construction above provides an abundance of examples of hyperbolic manifolds of simplest type that embed geodesically. This naturally suggests to go the opposite way: we start with an arithmetic $n$-manifold of simplest type, and we want to realise it as an embedded totally geodesic hypersurface in an $(n+1)$-arithmetic manifold of simplest type. The second result shows that this can often be done \cite{KRS}.

\begin{thm}\label{thm:embedding-arithmetic}
Let $M=\mathbb{H}^n/\Gamma$ be an arithmetic manifold of simplest type whose form $f$ is defined over a field $k$. If $\Gamma<\mathrm{O}(f, k)$ then, for any positive $q \in \mathbb{Q}$, the manifold $M$ embeds geodesically in an arithmetic manifold $X=\mathbb{H}^{n+1}/\Lambda$ of simplest type with form $f\oplus\langle q\rangle$ and $\Lambda<\mathrm{O}(f\oplus\langle q\rangle, k)$. If $M$ is orientable, $X$ can be chosen to be orientable. 
\end{thm}

The technical point of the statement is that the fundamental group $\Gamma$ of $M$ is required to be contained in the group $\mathrm{O}(f,k)$ of $k$-points of $\mathrm{O}(f, \mathbb{R})$. However, this is not too restrictive. In even dimensions, all hyperbolic arithmetic lattices are of simplest type, and lie in the group of $k$-points of the corresponding orthogonal group (c.f. \cite{Bo} and \cite[Lemma 4.2]{ERT}). In odd dimensions, if $\Gamma<\mathrm{O}(f, \mathbb{R})$ is arithmetic of simplest type, then the subgroup $\Gamma^{(2)}=\langle\gamma^2\,|\,\gamma \in \Gamma\rangle$ has finite index in $\Gamma$, and is contained in the group of $k$-points $\mathrm{O}(f, k)$. Therefore, at worst a finite-index Abelian cover $\mathbb{H}^{n}/\Gamma^{(2)}$ of $\mathbb{H}^{n}/\Gamma$ embeds geodesically.

\section{Embedding relative to hypersurfaces}\label{sec:embedding-boundary}

The goal of this section is to prove the following:

\begin{prop}\label{prop:main}
Let $M=\mathbb{H}^n/\Gamma$ be an arithmetic manifold of simplest type whose form $f$ is defined over $k$, and let $\mathscr S = \{ S_1, \dots, S_m \}$ be a finite collection of pairwise disjoint, properly embedded, totally geodesic hypersurfaces of $M$. 

If $\Gamma<\mathrm{O}(f, k)$, then $M$ embeds geodesically in a hyperbolic $(n+1)$-manifold $X$ containing $m$ disjoint, properly embedded, totally geodesic hypersurfaces $Y_1$, $\dots$, $Y_m$ that intersect $M$ orthogonally, with $Y_i\cap M = S_i$ for all $i=1, \dots, m$. If $M$ is orientable, $X$ can be chosen to be orientable.
\end{prop}

We prove Proposition \ref{prop:main} below in Section \ref{sec:setup}, assuming a technical lemma whose proof is postponed to Section \ref{sec:geom_fin}.

\subsection{Proof of Proposition \ref{prop:main}} \label{sec:setup}

By Theorem \ref{thm:embedding-arithmetic}, $M=\mathbb{H}^n/\Gamma$ embeds geodesically in $X_{\Lambda} = \mathbb{H}^{n+1}/\Lambda$, for a torsion-free arithmetic lattice $\Lambda<\mathrm{O}(g, k)$ of simplest type such that $\Gamma<\Lambda$, with $g=f\oplus\langle q\rangle$ for a positive $q \in \mathbb{Q}$. 

We shall need more control on the embedding in the subsequent proof (c.f. also Remark~\ref{remark:choice-independent group}),  and thus pass to a finite-index subgroup $L<\Lambda$ such that $\Gamma<L$, satisfying some additional properties described in the sequel. 

For any finite-index subgroup $L < \Lambda$ such that $\Gamma < L$, let $\pi_L\colon\matH^{n+1}\to \matH^{n+1}/L = X_L$ denote the canonical projection. We call \emph{horizontal hyperplane} the $L$-hyperplane $H$ of $\matH^{n+1}$ corresponding to the space-like vector $(0,\ldots,0,1)$ in the quadratic space $(\mathbb{R}^{n+2},g)$. The group $\Gamma<L$ now acts on all $\mathbb{H}^{n+1}$, preserving the hyperplane $H \cong \mathbb{H}^n$ without exchanging its two sides.  We have $\Stab_L(H)=\Gamma$, and we call $M=H/\Gamma\subset X_L$ the \emph{horizontal hypersurface} of $X_L$. 

For each hypersurface $S \in \mathscr S$ of $M$, we now choose a $\Gamma$-hyperplane $H_{v}$ of $H$ for an appropriate space-like vector $v$ in the quadratic space $(\mathbb{R}^{n+1},f)$ such that $H_{v}$ projects to $S \subset M$. Notice, that such $v$ is not unique, while any two choices differ only by an element of $\Gamma$. Now interpret each $v$ as a space-like vector in the quadratic space $(\mathbb{R}^{n+2},g)$. We call the corresponding $L$-hyperplane $V\subset\matH^{n+1}$ a \emph{vertical hyperplane}, and $Y = \pi_L(V)$ a \emph{vertical hypersurface} of $X_L$. Let $\mathscr Y = \{ Y_1, \ldots, Y_m \}$ be the collection of vertical hypersurfaces, with $Y_i$ associated with $S_i$ for each $i = 1,\ldots, m$. 

\begin{figure}
\centering
\includegraphics[width = 14 cm]{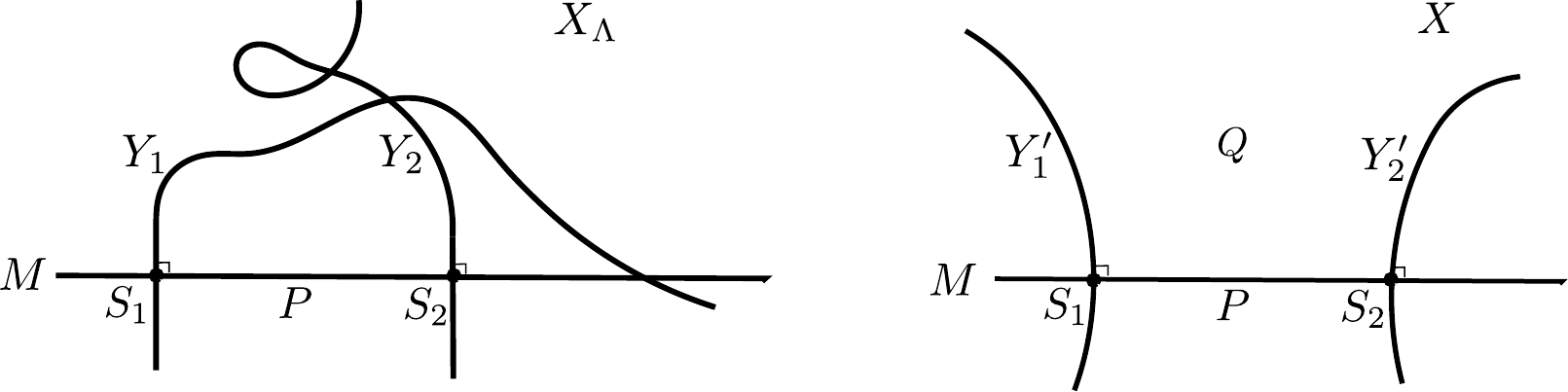}
\caption{\footnotesize A schematic picture of $X_\Lambda$ (left) and of its finite cover $X$ (right) with their horizontal and vertical hypersurfaces. A piece $P$ of $M$ embeds ``nicely'' in a piece $Q$ of $X$.  Each $Y'_i$ finitely covers $Y_i$.}
\label{fig:proof}
\end{figure}

Each $Y \in \mathscr Y$ is the image of an immersion $\iota\colon V/\mathrm{Stab}_L(V)\to X_L$, which is not necessarily an embedding. Note also that the vertical hypersurfaces are not necessarily pairwise disjoint. Moreover, since $V \perp H$ and $\pi_L(V \cap H) = S$, each $Y$ intersects $M$ orthogonally in the corresponding $S$, but there might be other intersections in $X_L$ between $Y$ and $M$, or between two distinct vertical hypersurfaces (c.f. Figure \ref{fig:proof}, left).

Our goal is to produce a finite index subgroup $L < \Lambda$ such that the following properties hold (c.f. Figure \ref{fig:proof}, right):
\begin{enumerate}
\item the group $L$ contains $\Gamma$ (so that $M$ lifts to $X_L$, as already mentioned);
\item the vertical hypersurfaces of $X_L$ are all embedded and pairwise disjoint; 
\item the intersection $Y_i \cap M$ equals $S_i$ for all $i = 1, \ldots, m$.
\end{enumerate}

Let $G_L$ denote the subgroup of $L$ generated by the stabiliser $\Gamma=\mathrm{Stab}_L(H)$ of the horizontal hyperplane, together with the stabilisers $\mathrm{Stab}_L(V_1)$, $\ldots$, $\mathrm{Stab}_L(V_m)$ of the vertical hyperplanes. 

\begin{rem}\label{remark:choice-independent group}
Since $\Gamma < L$, the group $G_L$ is independent of the particular choice of the vertical hyperplanes $V_1, \dots, V_m$ with $\pi_L(V_i)=Y_i$. Indeed, if $\pi_L(V)=\pi_L(V') \in \mathscr Y$ then $\gamma(V)=V'$ for some $\gamma \in \Gamma$. Therefore $\Stab_L(V)$ and $\Stab_L(V')$ are conjugate by $\gamma$, and they generate the same subgroup together with $\Gamma$.
\end{rem}

Consider now the abstract glueing
\begin{equation} \label{eq:abstract_glueing} 
A = M\,\cup_{S_1}\,V_1/\mathrm{Stab}_L(V_1)\,\cup_{S_2}\,\ldots\,\cup_{S_m}\,V_m/\mathrm{Stab}_L(V_m).
\end{equation}
The following lemma (which reminds of the Klein--Maskit combination theorem, c.f. also \cite[Lemma 7.1]{BHW}) will be proved in Section \ref{sec:geom_fin} by applying Poincar\'e's fundamental polyhedron theorem:

\begin{lemma}\label{lem:geometrically_finite}
There exists a finite-index subgroup $L<\Lambda$, with $\Gamma<L$, such that $G_L$ is geometrically finite and $A$ embeds $\pi_1$-injectively into $X_L$ with fundamental group $G_L$.
\end{lemma}

Given that $G_L$ is geometrically finite, Theorem \ref{thm:GFERF} implies that $G_L$ is separable in $L$. A separability argument due to Scott \cite{S} implies that there exists a finite cover $X \to X_L$ such that $A$ embeds in $X$ as follows: $M$ and each $V_i/\Stab_L(V_i)$ is a totally geodesic hypersurface of $X$, each $V_i/\Stab_L(V_i)$ intersects $M$ along $S_i$ orthogonally and any two distinct hypersurfaces of the form $V_i/\Stab_L(V_i)$, $i=1,\dots,m$, are disjoint. Thus, the proof of Proposition \ref{prop:main} is complete up to Lemma \ref{lem:geometrically_finite}.

In order to prove Lemma \ref{lem:geometrically_finite}, we will find $L < \Lambda$ such that $X_L$ can be obtained by pairing a finite number of thick convex cells in $\matH^{n+1}$ isometrically along their facets, with each cell having a finite number of facets. This easily implies that the group $L$ admits a finite-sided fundamental polyhedron, and is therefore geometrically finite. 
These cells will be obtained from the Vorono\"i decomposition of $X_L$ with respect to an appropriate choice of a finite set of points in the horizontal hypersurface $M$, as we now explain.

\subsection{Relative Vorono\"i decompositions} \label{sec:adm}

Recall that, given a metric space $(X,d)$ and a collection $\mathscr X$ of points, the \emph{Vorono\"i decomposition} of $X$ with respect to $\mathscr X$ is the decomposition of $X$ into \emph{cells} $C_p = \{ x \in X \, | \, d(x,p) \leq d(x,q) \ \forall q \in \mathscr{X}, q \neq p \}$, $p \in \mathscr X$.

Let $M = \matH^n / \Gamma$ be a hyperbolic manifold, $\pi_\Gamma \colon \matH^n \to M$ the canonical projection, $\mathscr X \subset M$ 
a finite set, and $\tilde{\mathscr X} = \pi^{-1}_\Gamma(\mathscr X)$. Consider the Vorono\"i decompositions of $M$ and $\matH^n$ with respect to $\mathscr X$ and $\tilde{\mathscr X}$, respectively. Each cell $C \subset \matH^n$ of the decomposition is a convex $n$-polytope which projects down to a cell $\pi_\Gamma(C)$ of $M$. There is a unique $x \in \tilde{\mathscr X}$ such that $x \in C$, called the \emph{centre} of $C$. Similarly, $\pi_\Gamma(x)$ is the \emph{centre} of $\pi_\Gamma(C)$.

A finite-sided fundamental domain $D_M \subset \matH^n$ for the action of $\Gamma$ can be constructed by pairing together a finite number of such cells of $\matH^n$ isometrically along some of their facets. This domain naturally satisfies the hypothesis of Poincar\'e's fundamental polyhedron theorem; c.f. \cite{EP} for a detailed exposition. 

\begin{figure}
\centering
\includegraphics[width = 10 cm]{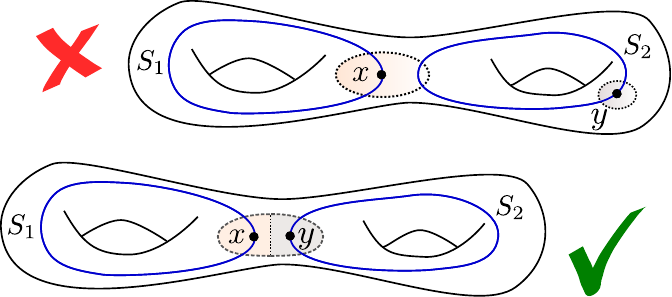}
\caption{\footnotesize A non-admissible choice for a pair of points on a surface with two disjoint geodesics (top), compared to an admissible choice (bottom). In the non-admissible case, there are points of $S_2$ which are closer to $x \in S_1$ rather than $y \in S_2$. In the admissible case, all points of $S_1$ are closer to $x$ and all points of $S_2$ are closer to $y$.}
\label{fig:admissible_vs_not}
\end{figure}

Let now $\mathscr S =\{S_1,\dots,S_m\}$
be a finite collection of pairwise disjoint, properly embedded, totally geodesic hypersurfaces of $M$. We will be interested in Vorono\"i decompositions that are ``coherent'' with respect to $\mathscr S$.

\begin{defn}\label{def:admissible}
Given $M$ and $\mathscr S$ as above, we say that a finite set $\mathscr X \subset S_1 \cup \ldots \cup S_m$ is \emph{admissible} with respect to $\mathscr S$ if in the Vorono\"i decomposition of $M$ associated with $\mathscr X$ each $S \in \mathscr S$ is covered only by the cells whose centres lie in $S$.
\end{defn}

In other words, we require the Vorono\"i decomposition of each $S \in \mathscr S$ with respect to $\mathscr X \cap S$ to coincide with the induced decomposition obtained by intersecting the cells of the Vorono\"i decomposition of $M$ with $S$. 

A random choice of $\mathscr X$ may not be admissible, as shown in Figure \ref{fig:admissible_vs_not}.

\begin{claim}\label{claim:admissible}
Given $M$ and $\mathscr S$ as above, with $M$ of finite volume, there exists an admissible set $\mathscr X$.
\end{claim}

\begin{proof}
First assume that $M$ is compact, which implies that the hypersurfaces $S_i$ are compact as well. For each $i=1,\dots,m$, let $\delta_i>0$ be the minimum distance between $S_i$ and $\bigcup_{j \neq i} S_j$. The set of open balls $\{B_x(\delta/2):\,x\in S_i\}$ is an open covering of $S_i$. By compactness, we can extract a finite cover, which gives a finite set $\mathscr X_i = \{x_1, \dots, x_{m_i}\} \subset S_i$ whose $\delta/2$-neighbourhoods cover $S_i$. The set $\mathscr X = \bigcup_i \mathscr X_i$ of points of $M$ is obviously admissible with respect to $\mathscr S$. 

If $M$ is non-compact, some $\delta_i$ might be zero. In this case, we change our argument as follows: truncate the cusps of $M$, so that $M$ decomposes as the union of a compact set $M_{\mathrm c}$ and a finite number of cusps, each of the form $E\times [0,\infty)$, where $E$ is a compact Euclidean manifold (the section of the corresponding cusp). In this way, each $S_i$ is similarly decomposed as the union of a compact set $S_i\cap M_{\mathrm c}$ and a finite number of cusps (possibly none).

For every $i$, apply the above argument of the compact case to each Euclidean cusp section $E$ of $M$ which $S_i$ intersects, in order to obtain a finite set of points in $E \times \{0\}$. The key property here is that, in the Vorono\"i decomposition of $E$ with respect to $\mathscr S\cap (E\times \{0\})$, each set of the form $S_i\cap (E\times \{0\})$ will be covered by cells centred on $S_i$. Let $\mathscr X'_i$ be the set of all points in $S_i$ obtained in this way over all cusps of $M$. 

Now that we have dealt with the ends of $M$, we turn to the compact part $M_{\mathrm c}$. Apply once more the same argument of the compact case to each set of the form $S_i \cap M_{\mathrm c}$ in order to obtain another finite set of points $\mathscr X''_i$ in $S_i \cap M_{\mathrm c}$. Set $\mathscr X_i = \mathscr X'_i \cup \mathscr X''_i$, and $\mathscr X = \bigcup_i \mathscr X_i$. The latter is admissible with respect to $\mathscr S$: the ends of each $S_i$ are covered by the cells centred on $\mathscr X'_i$, while $ S_i \cap M_{\mathrm c}$ are covered by the cells centred on $\mathscr X''_i$.
\end{proof}

\begin{figure}
 \begin{center}
  \includegraphics[width = 7.5 cm]{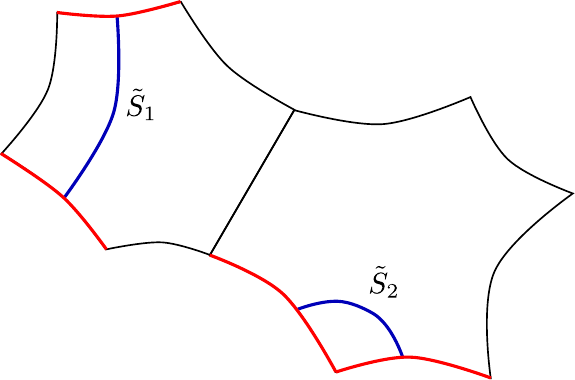}
 \end{center}
 \caption{\footnotesize Partitioning the facets of a tessellation of $\matH^2$ into two types, with those of the first type coloured red. The intersection of the cells with the lifts of all the $S_i$'s is coloured blue.} \label{fig:tessellation face type}
\end{figure} 

Given a complete hyperbolic manifold $M$ with a collection $\mathscr S$ of hypersurfaces and an admissible set $\mathscr X$ as above, it will be useful to partition the facets of the cells of $\matH^n$ into two types (c.f. Figure \ref{fig:tessellation face type}):

\begin{defn}
Let $C$ be a cell of the Vorono\"i decomposition of $\matH^n$ associated with $\mathscr X$, and $F$ be a facet of $C$. We say that $F$ is of the \emph{first type} if it intersects a lift $\tilde{S}$ of some $S \in \mathscr S$.  The facets of $C$ that are not of the first type are called facets of the \emph{second type}. 
\end{defn}
The same terminology is adopted for the bounding hyperplanes of $C$, depending on the type of the facet they contain.

\subsection{Nestedness of bounding hyperplanes} \label{sec:nested}

Let $M = \matH^n/\Gamma$ be a hyperbolic manifold, and consider the Vorono\"i decomposition of $\matH^n$ associated with some finite set $\mathscr X \subset M$. Given two discrete subgroups $G, G' < \Gamma$, let $C,C'$ be two cells centred at the same point $x \in \tilde{\mathscr X} \subset \matH^n$, and $B,B'$ two disjoint bounding hyperplanes for any of $C$ or $C'$. There is a unique halfspace $\calH$ (resp. $\calH'$) bounded by $B$ (resp. $B'$) containing $x$.

We say that $B$ and $B'$ are \emph{nested} if either $\calH \subset \calH'$ or $\calH' \subset \calH$. The halfspaces $\calH$ and $\calH'$ cannot be disjoint, since both of them contain $x$. Clearly, if both $B$ and $B'$ bound the same cell $C$, they are not nested. Otherwise, it is not clear a priori whether $B$ and $B'$ are nested or not.

This is best explained by considering the simple case where $H$ and $V$ are two non-orthogonal geodesics in $\mathbb{H}^2$ intersecting in a point $x$, as shown in Figure \ref{fig:nested-vs-not-nested}, with each of the stabilisers $G_H$ and $G_V$ of $H$ and $V$ respectively generated by a hyperbolic translation.

The endpoints of the geodesic $V$ can lie outside of the Vorono\"i domain (centred at $x$) for the translation $H$ (c.f. Figure \ref{fig:nested-vs-not-nested}, left). If the translation length along $V$ is chosen large enough, the fundamental domain $C_H$ of $G_H$ ends up being contained in the fundamental domain $C_V$ of $G_V$. In this situation some of the bounding hyperplanes for the two domains are nested.

\begin{figure}
\centering
\includegraphics[width = 10 cm]{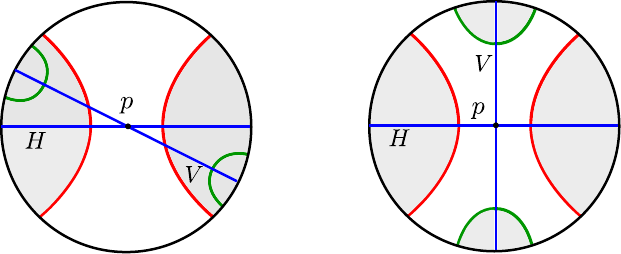}
\caption{\footnotesize The fundamental domains $C_H$ and $C_V$ for two translations along two geodesics $H$ and $V$ in the plane. The bounding hyperplanes for $C_H$ are drawn in red while those for $C_V$ are drawn in green. On the left, $H$ and $V$ are not orthogonal, and the two fundamental domains are nested. The intersection $C_H \cap C_V = C_H$ is not a fundamental domain for $G = \langle G_H, G_V \rangle$. On the right, the $H$ and $V$ are orthogonal, so there is no nesting and $C_H \cap C_V$ is a fundamental domain for $G$.}
\label{fig:nested-vs-not-nested}
\end{figure}

If $V$ and $H$ are orthogonal (c.f. Figure \ref{fig:nested-vs-not-nested}, right), no matter how short the translation along $H$ is, if the translation length along $V$ is chosen large enough the bounding hyperplanes of $C_H$ and $C_V$ are disjoint, and $C_H \cap C_V$ is indeed a fundamental domain for the group $G$ generated by the two translations (a free group on two generators). 

If we allow arbitrarily large translation length along $H$, then nesting phenomena can be avoided even if $V$ and $H$ are not orthogonal, simply because the endpoints of the $H$ and $V$ are distinct. However, in what follow we will rather forego taking subgroups of $G_H$.

\subsection{Proof of Lemma \ref{lem:geometrically_finite}  } \label{sec:geom_fin}

Let us fix $R>0$. Since $\Gamma$ is separable in $\Lambda$, there exists a finite-index subgroup $L_R<\Lambda$ containing $\Gamma$ such that every hyperbolic element $g\in L_R\smallsetminus\Gamma$ has translation length greater than $R$.

If $X_\Lambda$ and $X_{L_R}$ are non compact, fix also an arbitrary truncation of the cusps of $X_{L_R}$. We can furthermore require that any parabolic element in $g\in L_R\smallsetminus\Gamma$ has Euclidean translation length (relative to the chosen cusp truncation) greater than $R$. 

Let us fix a set $\mathscr X$ of points in the horizontal hypersurface $M$ of $X_{L_R}$ which is admissible with respect to $\mathscr S=\{S_1, \dots, S_m\}$. Consider the Vorono\"i decomposition of the horizontal hyperplane $H\subset \mathbb{H}^{n+1}$ into $n$-dimensional convex cells, which is clearly preserved by the action of $\mathrm{Stab}_{{L_R}}(H)$. A (possibly disconnected) fundamental domain $D_M \subset H$ for the action of $\Gamma$ on $H$ can be obtained by considering a set of cells of the Vorono\"i decomposition of $H$.

Now, we extend orthogonally each bounding hyperplane of the decomposition of $H$ to a hyperplane in $\matH^{n+1}$, to get a decomposition of $\mathbb{H}^{n+1}$ into $(n+1)$-dimensional convex cells. The fundamental domain $D_M\subset H$ extends to a finite-sided fundamental domain $\calD_M\subset \mathbb{H}^{n+1}$ for the action of $\mathrm{Stab}_{{L_R}}(H)$ on $\mathbb{H}^{n+1}$. The facets of $\calD_M$ can be partitioned into two types, which they inherit from those of $D_M$.

The discussion above applies similarly to the vertical hypersurfaces as follows. Let $Y \in \mathscr Y$ correspond to $S \in \mathscr S$. The set $\mathscr X \cap S$ is admissible for $Y$ with respect to $\{S\}$. Consider the associated Vorono\"i decomposition of a vertical hyperplane $V$ associated with $Y$. We can again build a fundamental domain $D_{Y}$ for the action of $\mathrm{Stab}_{{L_R}}(V)$ on $V$ consisting of cells of the Vorono\"i decomposition of $V$, and extend it to a fundamental domain $\calD_{Y}$ for the action of $\mathrm{Stab}_{{L_R}}(V)$ on $\matH^{n+1}$. We do this in the following way: for each cell of $\calD_M$ centred at a point $x \in  \tilde{S}$, we require the cell of the Vorono\"i decomposition of $V$ centred at $x$ to belong to $\calD_{Y}$.  By doing so, we obtain a one-to-one correspondence between the cells of the domain $\calD_{Y}$ and the cells of $\calD_M$ whose centres project down to $S$. Two corresponding cells $\calC \subset \calD_M$ and $\calC' \subset \calD_{Y}$ are centred at the same point $x \in \mathbb{H}^{n+1}$.

Notice that all bounding hyperplanes of the first type for the cells of $\calD_{Y}$ are also bounding hyperplanes of the first type for the cells of $\calD_M$. The pairing maps between such facets are the same both when viewed as facets of $\calD_M$ and of $\calD_{Y}$. The finite set of convex cells obtained by considering \emph{only} the halfspaces bounded by the hyperplanes of the first type is a fundamental domain for the action of the group $\mathrm{Stab}_{{L_R}}(\tilde{S})$ on $\mathbb{H}^{n+1}$. This group is clearly a subgroup of both $\mathrm{Stab}_{{L_R}}(H)$ and $\mathrm{Stab}_{{L_R}}(V)$.

This fact is the whole purpose of our careful choice of the admissible set of points $\mathscr X$: the Vorono\"i decompositions of $M$ and $Y$ agree on the corresponding $S$. As a consequence the fundamental domains for the associated groups of isometries of $\mathbb{H}^{n+1}$ share the respective facets, and the pairing maps on these facets agree. 

Since the vertical hyperplanes are pairwise disjoint and all orthogonal to the horizontal hyperplane $H$, there exists $R_0>0$ and a lattice $L = L_{R_0}$ such that the following holds: if a bounding hyperplane $B'$ of a cell $\calC'$ of $\calD_{Y}$ intersects a bounding hyperplane $B$ of a cell $\calC$ of $\calD_M$, then $B'$ is itself of the first type and is therefore a bounding hyperplane of $\calC$. The bounding hyperplanes of the second type for the cells of $\calD_{Y}$ are disjoint from those of the cells of $\calD_M$. Moreover, the bounding hyperplanes of the first type for the cells of $\calD_{Y_i}$ and $\calD_{Y_j}$ are disjoint whenever $i\neq j$.

Let us prove the latter fact. Given a subset $F$ of $\matH^n$, we denote by $\partial_\infty F$ its boundary at infinity, that is the intersection of the closure of $F$ in $\overline\matH^n = \matH^n \cup \partial_\infty \matH^n$ with $\partial_\infty \matH^n$.  Consider a cell $\calC$ for the domain $\calD_{M}$, obtained by extending orthogonally an $n$-dimensional cell $C$ of $H$. Then $\partial_\infty \calC$ consists of two conformal copies $C_1, C_2$ of the cell $C$. The closest point projection $C_j \to C$ is indeed conformal. The image in $C_j$ of $C\cap \tilde{S}$ lies on $\partial_\infty V$, for some $V$ which projects down to the vertical hypersurface $Y$ associated with $S$. Because of this, we see that $\partial_\infty V$ is disjoint from the boundary at infinity of the facets of the second type of $\calC$. This property holds true only because of orthogonality between the hyperplanes $H$ and $V$.

The boundary at infinity of each bounding hyperplane of a cell of the domain $\calD_{Y}$ is a conformal $(n-1)$-sphere in $\partial_\infty \mathbb{H}^{n+1}$ centred on $\partial_\infty V$. As $R \rightarrow \infty$, these spheres remain the same if they correspond to facets of the first type, while become arbitrarily small if they correspond to facets of the second type. At the same time, the cells of the domain $\calD_M$ don't change. When the spheres become small enough, the facets of the second type of $\calD_{Y}$ become disjoint from those of $\calD_M$, as shown in Figure \ref{fig:circle_shrink}. Also, since $V_1, \ldots, V_m$ are all orthogonal to $H$, then $\partial_\infty V_1, \ldots, \partial_\infty V_m$ can touch only in $\partial_\infty H$. Therefore, also the boundary hyperplanes of the first type for $\calD_{Y_i}$ and $\calD_{Y_j}$, $i\neq j$, eventually become disjoint.

This shows that, up to an appropriate choice of $L = L_{R_0}$ with sufficiently large $R_0$, the only intersections between the bounding hyperplanes for any of the domains $\calD_M$ or $\calD_{Y_i}$ will either happen between the bounding hyperplanes belonging to cells in a single fundamental domain, or between the bounding hyperplanes of the first type belonging to a cell of $\calD_M$ and a cell in one of the domains $\calD_{Y_i}$ (and in this case the bounding hyperplanes will coincide).

\begin{figure}
\centering
\includegraphics[width = 10 cm]{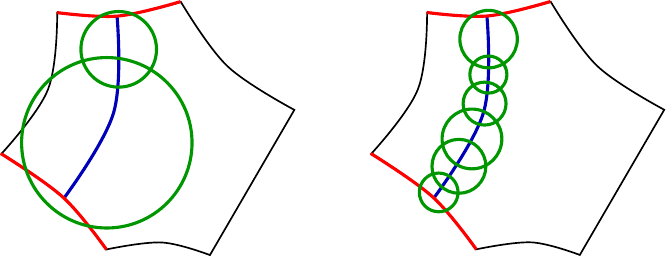}
\caption{\footnotesize A component of the boundary at infinity of a cell $\calC$ of $\calD_M$ for two different values of $R$ (increasing from left to right). Its facets of the first type are drawn in red. Its intersection with $\partial_\infty V$ is drawn in blue. The boundary at infinity of the supporting hyperplanes of the second type of a cell $\calC'$ of $\calD_{Y}$ with the same centre of $\calC$ are drawn in green. As $R$ increases, the green spheres became smaller and smaller, and their number increases. They eventually become disjoint from the facets of $\partial_\infty \calC$ of the second type.}
\label{fig:circle_shrink}
\end{figure} 
 
Now we prove that since the hyperplane $H$ is orthogonal to each $V_i$, there is no nesting between the bounding hyperplanes for a cell $\calC$ of $\calD_M$ and a cell $\calC'$ of $\calD_{Y_i}$ with the same centre. 
Indeed, given a bounding hyperplane $B$ of $\calC'$, the centre $y$ of the conformal sphere $\partial_\infty B$ belongs to $\partial_\infty V_i$, and projects down to a point in the interior of $\calC$ under the closest point projection. More importantly, $y$ is contained in one of the two components of $\partial_\infty \calC$, and this guarantees that $B$ and any of the bounding hyperplanes of $\calC$ are not nested.

Finally, consider the domain $$\calD = \calD_M \cap \calD_{Y_1} \cap \ldots \cap \calD_{Y_m} \subset \matH^{n+1}.$$ Each cell of $\calD$ is the intersection of a cell $\calC \subset \calD_M$ with a cell $\calC' \subset D_{Y_i}$, with $\calC$ and $\calC'$ centred at a common point $x \in \mathbb{H}^{n+1}$. The domain $\calD$ satisfies the hypothesis of Poincar\'e's fundamental polyhedron theorem, since each of the domains $\calD_M$ and $\calD_{Y_i}$ individually does, and there are no intersections between the hyperplanes of the second type. Since disjoint bounding hyperplanes for $\calD_M$ and $\calD_{Y_i}$ are not nested, all the pairing maps for the facets of these domains (which generate $\Gamma=\mathrm{Stab}_L(H)$ and $\mathrm{Stab}_L(V_i)$) survive as pairing maps between the facets of $\calD$.

The domain $\calD$ is therefore a fundamental domain for $G_L$, which is isomorphic to the amalgamated free product 
$$\pi_1(M)*_{\pi_1(S_1)}\pi_1(Y_1)*_{\pi_1(S_2)}\ldots *_{\pi_1(S_m)}\pi_1(Y_m).$$

Since $\calD$ is finite-sided, $G_L$ is geometrically finite, and the proof of Lemma \ref{lem:geometrically_finite} is complete.

\section{Proofs of the main theorems}\label{sec:proof-main}

We are ready to prove Theorems \ref{thm:main} and \ref{thm:counting} opening this paper. As the main result of the paper (Theorem~\ref{thm:main}) is established, it will follow that there are ``super-exponentially many'' geometrically bounding manifolds and their commensurability classes with respect to volume (Theorem \ref{thm:counting}). 

\subsection{Proof of Theorem \ref{thm:main} (embedding hyperbolic glueings)}

Let $M=P_1\cup\ldots\cup P_s$ satisfy the hypotheses of Theorem \ref{thm:main}. Each piece $P_j$ is obtained from some hyperbolic $n$-manifold $M_j$ of simplest type by cutting it open along a finite collection $\mathscr S_j$ of pairwise disjoint, totally geodesic hypersurfaces. Each manifold $M_j$ is arithmetic of simplest type with associated quadratic form $f_j$ defined over $k$.

By Proposition \ref{prop:main}, the manifold $M_j$ embeds geodesically in a hyperbolic $(n+1)$-manifold $X_j = \matH^{n+1}/L_j$ which contains a finite collection $\mathscr Y_j$ of properly embedded, pairwise disjoint totally geodesic ``vertical'' hypersurfaces, with each $Y \in \mathscr Y_j$ intersecting $M_j$ orthogonally in the corresponding $S \in \mathscr S_j$. We choose $X_j$ to be arithmetic of simplest type with associated form $g_j=f_j\oplus\langle q\rangle$ for some positive $q\in \mathbb{Q}$ which does not depend on $j$.

By cutting each $X_j$ open along the vertical hypersurfaces, we obtain some $(n+1)$-dimensional pieces $Q_1,\dots,Q_s$ so that $P_j$ is a totally geodesic hypersurface of $Q_j$ orthogonal to $\partial Q_j$, with $P_j \cap \partial Q_j = \partial P_j$. Each boundary component of $Q_j$ is either isometric to a vertical hypersurface $Y$, if $Y$ is two-sided in $X_j$, or it is an index two cover of $Y$, otherwise. Without loss of generality, we can assume that all vertical hypersurfaces are two-sided. Recall the ``abstract glueing'' $A$~from~\eqref{eq:abstract_glueing} in Section \ref{sec:setup}. We can replace the fundamental groups of one-sided $Y_j$'s by the respective index-two subgroups, and embed the new amalgamated product in a finite cover of $X_j$ so that the vertical hypersurfaces be two-sided. 

Our goal is to show that it is possible to choose $X_j$ so that the pairing maps between the boundary components of $P_1, \ldots, P_s$ producing $M$ extend to isometries between the corresponding boundary components of $Q_1, \ldots, Q_s$. In this way, by glueing these new pieces back together we will obtain an $(n+1)$-manifold $X$ into which $M$ embeds geodesically. 

Note that the boundary components of $Q_1,\ldots,$ $Q_s$ are pairwise commensurable: in other words, they have a common finite cover. This holds since we extend all the quadratic forms $f_j$ using the same rational number $q$. Indeed, the forms associated with the glueing locus of the various blocks $P_j$ are pairwise projectively equivalent, and therefore so are their corresponding extensions. The latter define the commensurability classes of the connected components of $\partial Q_1,\ldots,$ $\partial Q_s$.

Now we can ensure that the respective boundary components of $Q_1,\ldots,$ $Q_s$ are pairwise isometric. In order to do this, introduce a new abstract glueing $A_j'$ obtained by attaching to $M_j$ along each $S \in \mathscr S_j$ the respective $V/\Delta$, where $\Delta$ corresponds to a common finite index cover for the two appropriate boundary components of $Q_j$ and $Q_l$. The fundamental group of $A'_j$ is geometrically finite and therefore separable in $L_j$. So $A'_j$ embeds in a finite cover $X'_j$ of $X_j$ corresponding to some finite index subgroup $L'_j<L_j$. 

By cutting each $X'_j$ open along its vertical hypersurfaces, we obtain a new collection of pieces $Q'_j$ such that the pairing maps between $P_1, \ldots, P_s$ extend to the pairing maps for $Q'_1, \ldots, Q'_s$, as each $M_j$ lifts to the respective $Q'_j$. By glueing these new blocks together with the found pairing map, and then doubling the resulting manifold with boundary (if such boundary is non-empty), we finally produce an $(n+1)$-dimensional hyperbolic manifold $X$ into which $M$ embeds geodesically. 

Assume now that $M$ is orientable. We still need some work to ensure that the manifold $X$ can be chosen to be orientable as well. Note that each piece $P_j$ is orientable. Let $C \subset \partial P_j$ be a boundary component of $P_j$, and $D \subset \partial Q_j$ the corresponding boundary component of $Q_j$, which contains $C$ as a totally geodesic submanifold. The piece $Q_j$ (and so $D$) can be chosen to be orientable. We furthermore require $D$ to admit an orientation reversing isometry which acts by fixing $C$ pointwise and exchanging its two sides. This can always be achieved up to considering an appropriate finite-index cover, as we now show. 

Let the vertical hypersurface $Y \in \mathscr Y_j$ of $X_j$ correspond to $D \subset \partial Q_j$, and $S \in \mathscr S_j$  correspond to $C \subset \partial D$. The hypersurfaces $M_j$ and $Y$ lift to two orthogonal hyperplanes $H$ and $V$, respectively, in the universal cover $\mathbb{H}^{n+1}$ of $X_j$, with $H\cap V$ corresponding to the universal cover of $S$. Let $g$ be the reflection in $H$, and let $\Delta < \Isom(\matH^{n+1})$ denote the fundamental group of $Y$, which acts on $\mathbb{H}^{n+1}$ by preserving $V$. Clearly $g$ fixes $H\cap V$ pointwise and preserves $V$.

We now consider the group $\Delta' = \Delta\, \cap\, g \Delta g^{-1}$. Since the hyperplane $H$ is a $k$-hyperplane, the reflection $g$ lies in $\mathrm{O}(f_j,k)$ and hence it commensurates $\Delta$. Therefore the group $\Delta'$ has finite index in $\Delta$ and we denote by $Y'$ the associated finite-index cover of $Y$. Since $g$ fixes $V\cap H$ pointwise, it commutes with all elements of $\pi_1(S)$ and therefore $S$ lifts to $Y'$. Moreover, $g$ normalises $\Delta'$ and this means that $g$ corresponds to an orientation-reversing isometric involution of $Y'$ which fixes the hypersurface $S$ pointwise, while exchanging its two sides. 

Thus, in the abstract glueing $A'_j$ we can require the vertical manifolds $V/\Delta$ to admit such an orientation-reversing involution. We can therefore freely prescribe the orientation class of the glueing maps between the boundary components of the pieces $Q_1,\dots,Q_s$, without changing the manifold $M$ which we wish to embed.

In particular, we can make the resulting manifold $X$ orientable, containing $M$ as a two-sided hypersurface, and the proof of Theorem \ref{thm:main} is complete.

\begin{rem}\label{rem:embedding-preserves-type}
 Our embedding procedure for non-arithmetic manifolds clearly preserves their type: Gromov--Pyatetski-Shapiro manifolds embed into Gromov--Pyatetski-Shapiro ma\-ni\-folds, and similarly for Agol--Belolipetsky--Thomson manifolds. 
\end{rem} 

\subsection{Proof of Theorem~\ref{thm:counting} (counting geometric boundaries)}

In this section we follow the idea by Gelander and Levit from \cite{GL}.

Let $k$ be either $\matQ$ or $\matQ(\sqrt2)$, depending on whether we want to consider cusped or closed manifolds, respectively. Consider the quadratic form
$$f_n(x) =
\begin{cases}
-2x^2_0 + x^2_1 + \ldots + x^2_n & \mbox{if}\ k = \matQ\,, \\
-\sqrt{2} x^2_0 + x^2_1 + \ldots + x^2_n & \mbox{if}\ k = \matQ(\sqrt2).
\end{cases}
$$
Then, let $f^{\rma^\pm}$, $f^{\rmb^\pm}$, $f^\rmu$ and $f^\rmv$ be six non-equivalent admissible quadratic forms 
$$f^\rmx(x) = f_{n-1}(x) + p_\rmx \cdot x^2_n,$$
over $k$, where $p_\rmx \in R_k$ is a prime and $\rmx$ is any of the six symbols $\rma^\pm, \rmb^\pm, \rmu, \rmv$.
There are infinitely many choices for such a collection of quadratic forms \cite[Lemma 4.11]{GL}. 

Now, let $S' = \matH^{n-1} / \Delta$ be a non-orientable arithmetic manifold of simplest type with associated form $f_{n-1}$ and $\Delta \subset \Or(f_{n-1},k)$. Notice that such a manifold $S'$ certainly exists. Indeed, the lattice $\Or(f_{n-1}, R_k)$ clearly contains orientation-reversing elements, such as the reflection in the orthogonal hyperplane to any space-like vector in the standard basis of $\mathbb{R}^{n+1}$. By \cite[Theorem 1.2]{LR2}, $\Or(f_{n-1},R_k)$ has a torsion-free subgroup $\Gamma'$ of finite index containing an orientation-reversing element. Take now $S$ to be the  orientation cover of the manifold $S' = \matH^{n-1} / \Gamma'$, and $\varphi$ an involution of $S$ such that $S' \cong S / \langle \varphi \rangle$.

\begin{prop} \label{prop:counting}
For each symbol $\rmx \in \{\rma^\pm, \rmb^\pm, \rmu, \rmv\}$, there exists an arithmetic manifold $M_\rmx = \matH^n / \Gamma_\rmx$ of simplest type with associated form $f^\rmx$ and $\Gamma_\rmx \subset \Or(f^\rmx, k)$, from which one can carve a piece $P_\rmx$ whose boundary consists of 2 (resp.\ 4) copies of $S$ if $\rmx \in \{\rma^\pm, \rmb^\pm\}$ (resp. $\rmx \in \{ \rmv, \rmu\}$). Moreover, the pieces of the form $P_{\rmv}$ can be chosen to be non-orientable.
\end{prop}

\begin{proof}
By Theorem \ref{thm:embedding-arithmetic}, for every $\rmx$ we can embed $S$ geodesically into some orientable arithmetic $M'_\rmx = \matH^n / \Gamma'_\rmx$ of simplest type with $\Gamma'_\rmx \subset \Or(f^\rmx, k)$. We now apply \cite[Proposition 4.3]{GL} in order to build orientable manifolds $M_\rmx$ such that:
\begin{enumerate}
\item if $\rmx \in \{ \rma^\pm, \rmb^\pm \}$, $M_\rmx$ contains a non-disconnecting copy of $S$;
\item the manifold $M_\rmu$ contains two disjoint copies of $S$ such that their union does not disconnect $M_\rmu$;
\item the manifold $M_\rmv$ contains three disjoint copies of $S$ such that their union does not disconnect $M_\rmv$.
\end{enumerate}
In order to build the pieces of the form $P_\rmx$ for $\rmx \in \{\rma^\pm, \rmb^\pm\}$, we simply cut open $M_\rmx$ along $S$. The resulting manifold has two totally geodesic boundary components. Similarly, in order to build $P_\rmu$ we cut $M_\rmu$ along the two copies of $S$, thus obtaining a piece with four boundary components.

Finally, we build $P_\rmv$ in two steps. We first cut $M_\rmv$ open along the three copies of $S$ in order to obtain a manifold with six totally geodesic boundary components, each isometric to $S$. We choose two boundary components which are the result of cutting along a single copy of $S$ in $M_\rmv$ and identify them isometrically using the orientation-reversing isometry $\varphi$ of $S$. By doing so, we obtain a non-orientable piece $P_\rmv$ with four boundary components, each isometric to $S$.
\end{proof}

Now, for every finite $4$-regular rooted simple graph with edges labelled by $\rma^\pm$ and $\rmb^\pm$, we put $P_\rmv$ at the root, $P_\rmu$ at all the other vertices, and $P_{\rmx}$ at each $\rmx$-labelled edge, whenever $\rmx \in \{ \rma^\pm, \rmb^\pm \}$.

After pairing isometrically the boundary components of the various pieces as prescribed by the graph (any identification of the boundary components of the pieces with the edges of the graph and any pairing isometry works), we get a hyperbolic manifold $M'$ with empty boundary. Such manifold is non-orientable because the piece $P_\rmv$ is non-orientable. 

As follows from the proof of \cite[Proposition 3.3]{GL}, for $m \in  \matZ$ large enough there are at least $m^{cm}$ such graph with at most $m$ vertices, so the number of manifolds $M'$ of volume $\leq v$ produced in this way is at least $v^{cv}$ for $v \gg 0$. These manifolds are pairwise incommensurable \cite[Section 4]{GL} and therefore so are their orientable double covers.

For any manifold $M'$ constructed above, we apply Theorem \ref{thm:main} to its orientation cover $M$, and embed it geodesically into some  orientable $X$. Since $M$ has an orientation-reversing fixed-point-free isometric involution, it also bounds geometrically. The proof of Theorem \ref{thm:counting} is complete.

\section{Manifolds that do not embed geodesically} \label{sec:rmk}

We conclude the paper with some additional observations on hyperbolic manifolds that do not embed geodesically.

It can be easily shown that not all hyperbolic surfaces embed geodesically. Indeed, consider a finite-area surface $S = \mathbb{H}^2/\Gamma_S$ that embeds totally geodesically into a finite-volume $3$-manifold $M = \mathbb{H}^3/\Gamma_M$. Up to conjugation, we can suppose that $\Gamma_S < \Gamma_M$, and thus for the trace fields we have $K_S = \mathbb{Q}[\mathrm{tr}\, \gamma : \gamma \in \Gamma_S] \subseteq \mathbb{Q}[\mathrm{tr}\, \gamma : \gamma \in \Gamma_M] = K_M$. 

As a consequence of the Mostow--Prasad rigidity, the trace field $K_M$ has to be an algebraic number field. However, it is not hard to produce a surface $S$ with $K_S$ being transcendental.  Nevertheless, as shown in \cite{FS}, those surfaces that embed geodesically form a countable dense subset of the moduli space.  

Except for the above, it is unknown if there exists an $n$-dimensional ($n\geq 3$) closed or cusped hyperbolic manifold that does not embed geodesically. 
Notice that there are hyperbolic manifolds which embed geodesically but do not bound geometrically. As suggested by Alan Reid to the authors, the Seifert--Weber dodecahedral space geodesically embeds by Theorem \ref{thm:embedding-arithmetic}, but has non-integral $\eta$-invariant and therefore does not bound geometrically by \cite{LR1}. In particular, Long and Reid's obstruction for bounding geometrically does not give any obstruction on embedding geodesically. See \cite{KRR} for similar examples of cusped 3- and 4-manifolds.

It would be also interesting to know if there exists a hyperbolic manifold without finite covers that embed geodesically or, conversely, if all hyperbolic manifolds do embed virtually.

In addition to the above list, at the moment we do not know if one of $\beta_3(v)$ or $B_3(v)$ is finite for $v$ sufficiently large (c.f. \cite[Question 1.6]{KR}). Recall that $\beta_3(v)$ denotes the number of $3$-dimensional hyperbolic geometric boundaries of volume $\leq v$ up to isometry, and $B_3(v)$ is the number of commensurability classes of $3$-dimensional hyperbolic geometric boundaries of volume $\leq v$.

\end{document}